\documentclass[11pt,a4paper]{amsart}

\usepackage{fullpage}

\usepackage{graphicx} 
\usepackage[utf8]{inputenc}
\usepackage[english]{babel}
\usepackage[T1]{fontenc}
\usepackage{amssymb}
\usepackage{mathtools}
\usepackage{amsfonts}
\usepackage{amsmath}
\usepackage{amsthm}
\usepackage{mathrsfs}
\usepackage{changepage}
\graphicspath{{images/}}
\usepackage{fancyhdr}
\usepackage{csquotes}

\usepackage[backend=biber, sortcites=true, maxbibnames=50]{biblatex}
\usepackage{hyperref}
\hypersetup{
	colorlinks=true,
	linkcolor=blue,
	citecolor=blue,
	urlcolor=blue,
	pdftitle={}
}

\newtheorem{theorem}{Theorem}  

\theoremstyle{plain}
\newtheorem{thm}{Theorem}
\newtheorem{lemma}[thm]{Lemma}
\newtheorem{corollary}[thm]{Corollary}

\newtheorem{prop}[thm]{Proposition}

\theoremstyle{definition}
\newtheorem{definition}[thm]{Definition}

\newtheorem{remark}[thm]{Remark}
\newtheorem{exmp}[thm]{Example}

\numberwithin{equation}{section}
\numberwithin{thm}{section}

\usepackage{setspace}

\theoremstyle{remark} 
\newtheorem*{ack}{Acknowledgements}

\title{Deformations of Lipschitz Homeomorphisms}
\author{Mohammad Alattar}
\date{\today}

\address[Alattar]{Department of Mathematical Sciences, Durham University, United Kingdom}
\email{{mohammad.al-attar@durham.ac.uk}}

\begin{document}

\begin{abstract}
 We obtain the Lipschitz analogues of the results Perelman used from Siebenmann's deformation of homeomorphism theory in his proof of the stability theorem. Consequently, we obtain the Lipschitz analogue of Perelman's gluing theorem. Moreover, we obtain the analogous  deformation theory but with tracking of the Lipschitz constants. 
\end{abstract}

\maketitle

\section{Introduction}

In the early 1970's, Siebenmann \cite{siebenmann1972deformation} introduced a class of spaces, termed \emph{locally cone-like spaces} (\emph{CS sets}). Such spaces generalize the notion of a manifold. In particular, such spaces are not necessarily manifolds, but can always be stratified into manifolds. Although such a space can be quite singular, its space of open embeddings behaves as though it were the space of open embeddings of a manifold. More precisely, Siebenmann introduced a general deformation rule $\mathcal{D}(X)$, satisfied by $CS$ sets, that enjoys numerous desirable properties. For instance, the deformation rule allows one to deform open embeddings lying in a small enough neighborhood of the inclusion in a manner akin to the manifold case. In particular, the deformation rule yields an isotopy theory that generalizes the Edwards and Kirby isotopy theory \cite{DeformationsofSpaces}. Indeed, Siebenmann's theory shows that the homeomorphism group of a reasonable compactum $X$ that is not necessarily a manifold, yet enjoys the rule $\mathcal{D}(X)$, is locally contractible (say, with the compact-open topology). The local contractibility of the homeomorphism group of a compact manifold is well known and was first proven by Cernavskii in 1968 \cite{Cernavskii}, and later reproven, in the early 1970's, by the use of the simple, yet powerful torus trick due to Kirby \cite{stablehomeos,DeformationsofSpaces}.

The idea behind showing that $CS$ sets satisfy the rule $\mathcal{D}(X)$ is as follows. One notices that $CS$ sets of a certain ``depth'' $d$ can be covered by open sets of depth at most $d$, each of which is ``equivalent'' to some $\mathbb{R}^m\times cL$, where $cL$ denotes the cone over $L$. Thus, it suffices to show, through induction, that the principle $\mathcal{D}(\mathbb{R}^m\times cL)$ holds. The key point being that one must deform embeddings on $\mathbb{R}^m\times cL$ appropriately. Siebenmann's deformation rule $\mathcal{D}(X)$, and more specifically, the theory of Edwards and Kirby, are  topological and, although they offer many desirable properties, they do not necessarily satisfy the condition that the perturbed map inherits the regularity properties of the initial map.  In the 1970's, shortly after the announcement of the celebrated theory of Edwards and Kirby \cite{DeformationsofSpaces}, Sullivan \cite{Sullivan}, using deep theory from etalé cohomology, offered, for Lipschitz manifolds, an alternative to the Kirby immersion device. Namely, by replacing the torus by a certain quotient of hyperbolic space, Sullivan obtained an analogue of the Edwards--Kirby theory on Lipschitz manifolds. Therefore, Sullivan's construction allowed for the introduction of a Lipschitz deformation rule $\mathcal{L}(X)$.

Examples of sets admitting a CS stratification that are not necessarily manifolds include Alexandrov spaces (see \cite{fujioka2024alexandrov} for a nice proof of this fact). Alexandrov spaces, introduced by Burago, Gromov and Perelman \cite{burago1992ad} in 1990 (see also \cite{Alexander-Kapovitch-Petrunin,Burago-Burago-Ivanov}) are metric generalizations of complete Riemannian manifolds with a uniform lower sectional curvature bound. Indeed, they have played a major role in metric and Riemannian geometry (see for example \cite{naberli, tuschmannkapovitchandpetrunin, kapovitchwilking, grovewilhelm,
 kapovitchmondino,   groveshiohama,   Lytchak-Nagano, LytchakNagano2, bruè2024topological,  brue-mondino-semola, kapovitch-lytchak-petrunin, Fernando3, Alexandrovmaximalradius,radiusspheretheorem, tadashi,finitenessdim4}).

By analyzing the local structure of Alexandrov spaces, Perelman in 1991 used Siebenmann's theory of deforming homeomorphisms to establish the celebrated stability theorem (see \cite{perelman1991alexandrov,Vitali}). The stability theorem asserts that if a given sequence of compact $n$-dimensional Alexandrov spaces $\{X_i\}_{i\in \mathbb{N}}$ with a uniform lower curvature bound Gromov--Hausdorff converges to another compact Alexandrov space $X$ with no collapse, then for all sufficiently large $i$, $X_i$ and $X$ are homeomorphic. It is claimed that Perelman proved a bi-Lipschitz analogue, though no such proof has been published. 

To generalize both Siebenmann's theory and the topological techniques used by Perelman in his stability theorem to the Lipschitz category, one must, first, have appropriate deformation principles. Second, one must deform, appropriately, (open) Lipschitz embeddings on $\mathbb{R}^m\times (cone)$. Both of these points have been addressed and have been answered in \cite{SiebenmannSullivan,Sullivan}. Lastly, one must obtain Lipschitz analogues of the relevant topological ingredients of Perelman's stability theorem: the isotopy extension theorem (Theorem 6.5 in \cite{siebenmann1972deformation}) and the union lemma (Lemma 6.9 in \cite{siebenmann1972deformation}). Indeed, a corollary of both is the important fibration theorem (Corollary 6.14 in \cite{siebenmann1972deformation} and Theorem $A$ in \cite{perelman1991alexandrov}) that a  closed topological submersion $p\colon E\rightarrow X$ with compact fibers, such that $\mathcal{D}(p^{-1}(x))$ is satisfied for all $x\in X$, and $p^{-1}(x)$ is a locally connected Hausdorff space, is a locally trivial fiber bundle.

Using the techniques in Siebenmann's paper \cite{siebenmann1972deformation}, Perelman obtained the fundamental gluing theorem (Theorem $B$ in \cite{perelman1991alexandrov} and Theorem 4.6 in \cite{Vitali}). The gluing theorem, roughly speaking, gives conditions as to when one can glue locally defined homeomorphisms near a global Hausdorff approximation to a homeomorphism that remains close to the approximation. Using reverse induction, and the notion of a ``frame'', Perelman proved the stability theorem by essentially reducing it to the gluing theorem. Indeed, Perelman reduced the problem of stability, which is a global problem, to a local problem. We note that in addition to the stability theorem, the topological gluing theorem has found applications in the theory of $\mathrm{RCD}$ spaces (Theorem 3.5 in \cite{kapovitchmondino}).

In this paper we will first prove the Lipschitz analogue of the isotopy extension result in \cite{siebenmann1972deformation}. Further, we will under a canonical adaptation of the rule $\mathcal{L}(X)$ (see Definition \ref{refined principle}), prove the Lipschitz analogue of the ``union lemma''. The union lemma (Lemma \ref{uniion lemma}), roughly speaking, ensures that under certain mild but rather important conditions, one can glue ``product charts'' on a neighborhood of a union. The union lemma has further applications other than Siebenmann's theory of deforming homeomorphisms (see for instance page 227 in \cite{kirby-siebenmann}). Using the new Lipschitz union lemma, we will show that provided that the fibers of a nice enough Lipschitz submersion satisfy the (adapted) Lipschitz deformation rule, then the Lipschitz submersion is indeed a locally trivial Lipschitz bundle. Moreover the adapted rule will yield a stronger and more desirable isotopy extension theorem (see Corollary \ref{strengthening}).

Our first theorem is a generalization of Siebenmann's isotopy extension principle (Theorem 6.5 in \cite{siebenmann1972deformation}).

\begin{theorem}
\label{lipschitz isotopy}
Assume $B$ is a locally connected metric space and $X$ is a metric space that is locally compact and locally connected. Assume $V$ is a metric space such that $f_t\colon V\rightarrow X$ for $t\in B$ is a continuous family of open Lipschitz embeddings. Assume further that $C$ is a closed subset of $V$ with compact boundary in $V$ and that for all $t\in B$, $f_t(C)$ is closed in $X$. If $\mathcal{L}(V-C)$ holds, then for each $b\in B$, there exists a neighborhood $N_b$ of $b$, and a family of Lipschitz isomorphisms $F_t\colon X\rightarrow X$ $(t\in N_b)$ such that $F_t\circ f_b=f_t$ near $C$. Further if $C$ is compact, then one can take $F_t$ so that $F_t\circ f_b=f_b$ away from a compact neighborhood of $C$.
\end{theorem}

In addition to Perelman's proof of the stability theorem, the topological isotopy extension theorem has found applications in K-theory (see section 1 in \cite{Weiss-Bruce}). 

Note that our terminology might differ from other uses  (see Definition \ref{definition imp}). In particular, when we say ``Lipschitz embeddings'', we mean embedding in the Lipschitz ''category''. In particular, there is more data than just a topological embedding that is Lipschitz. We do this so that we follow the conventions in the literature and thus make the paper easier to read.

Making the natural addendum to $\mathcal{L}(X)$ (see above), and denoting the refined deformation principle by $\mathcal{L}^{*}(X)$, we prove our second main result, a Lipschitz fibration theorem.

\begin{theorem}
\label{lipschitz submersions}
Let $p\colon E\rightarrow X$ be a Lipschitz submersion, where $X$ is locally compact and locally connected. Assume $F=p^{-1}(b)$, where $b=p(y)$, is locally connected , locally compact and $\mathcal{L}^{*}(F)$ holds. If $C$ is a compact subset of $F$. Then, there exists a neighborhood $U$ of $C$ and a Lipschitz product chart $f\colon U\times N\rightarrow E$ about $U$ for $p$. Hence if  $p$ is proper and for each $x\in X$, $p^{-1}(x)$ is locally connected and $\mathcal{L}^{*}(p^{-1}(x))$ holds true, then $p$ is a Lipschitz bundle map.
\end{theorem}

We note that the topological analogue of the above theorem is a crucial step in the proof that  Alexandrov spaces are locally conical  (Theorem 4.4  in \cite{Vitali}, Theorem A in \cite{perelman1991alexandrov}, Theorem 1.4 \cite{elements}).

Lastly, we prove the Lipschitz analogue of Perelman's gluing theorem (Theorem $B$ in \cite{perelman1991alexandrov}, Theorem 4.6 in \cite{Vitali}) under the rule $\mathcal{L}(X)$. Indeed, the topological gluing theorem is a crucial step in the topological stability theorem. 

In what follows, denote by $\chi\colon \mathbb{R}^{\geq 0}\rightarrow \mathbb{R}^{\geq 0}$ increasing continuous functions such that $\chi(0)=0$.

\begin{theorem}
[Lipschitz Gluing Theorem]
\label{Lipschitz Gluing Theorem}
Let $X$ be a compact metric space that is locally connected and such that $\mathcal{L}(X)$ holds. Assume that $X$ is covered by finitely many open sets $\{U_{\alpha} \}_{\alpha \in \mathfrak{A}}$. Given a function $\chi_0$, there exists a function $\chi$ (depends on $X$, the cover $\{U_{\alpha}\}$ and $\chi_0$) such that the following holds. 

Given a  $\chi_0$-connected, metric compactum $\tilde{X}$, covered by open sets $\{\tilde{U}_{\alpha}\}_{\alpha\in \mathfrak{A}}$ and $\theta\colon X\rightarrow \tilde{X}$ is a $\delta$-approximation $(\delta$ is sufficiently small) and $\varphi_{\alpha}\colon U_{\alpha}\rightarrow \tilde{U}_{\alpha}$ are Lipschitz isomorphisms $\delta$-close to $\theta$. Then, there exists a bi-Lipschitz homeomorphism $\theta'\colon X\rightarrow \tilde{X}$ that is $\chi(\delta)$-close to $\theta$.
\end{theorem}

Two important remarks are in order. First, in the statement of Perelman's topological gluing theorem, the condition "MCS" can be replaced by the more general assumption that the spaces under consideration satisfy Siebenmann's deformation principle. Therefore Theorem \ref{Lipschitz Gluing Theorem} is truly the Lipschitz analogue of Perelman's gluing theorem. Second,
the union lemma allowed Perelman to obtain the strong gluing theorem (Complement to Theorem $B$ \cite{perelman1991alexandrov}, and Theorem $4.10$ in \cite{Vitali}).

Our article is organized as follows: In section \ref{section preliminaries and lipschitz deformation principle}, we discuss the preliminaries, give examples, and then define the Lipschitz deformation principle. In particular, we discuss the existence and non-existence of various  Lipschitz structures on various spaces that admit different types of curvature bounds. We further include references which are related to the examples. In section \ref{section Lipschitz Isotopy theorem}, we discuss how to obtain a relative formulation of Siebenmann's deformation principle. Then, we prove the Lipschitz isotopy extension theorem. In section \ref{section Relative Deformations, Lipschitz Submersions and Lipschitz Union Lemma}, we discuss Lipschitz submersions and prove the Lipschitz Union lemma. In section \ref{section applications} we prove the Lipschitz analogue of Perelman's deformation lemma and the gluing theorem and give a few remarks concerning Alexander isotopies on cones. Finally, in Section \ref{Lipschitz constant control}, we quickly develop the results parallel to the results in earlier sections. However, in this case, we control the Lipschitz constants.


\begin{ack}
This paper is part of the author's PhD work. I would like to express my  gratitude to my advisor, Fernando Galaz-Garcia, and to Martin Kerin for their support and guidance. I am also sincerely grateful to Alexander Lytchak for his very valuable comments, advice, support, and for numerous discussions. I would also like to thank Luis Guijarro, Vitali Kapovitch, Wilderich Tuschmann, and Burkhard Wilking for their support, suggestions, and insights. I am also thankful to Mauricio Che, Mo Dick Wong, and Alpar Meszaros for their valuable comments during the Metric Geometry seminar at Durham University. My thanks also goes to Jaime Santos Rodriguez, Kohei Suzuki, Massoumeh Zarei, and Yanpeng Zhi for their valuable comments and for fruitful conversations during the preparation of this manuscript. Finally, I would like to extend my gratitude to Michael Weiss for his valuable comments and for explaining the details of an important variation of Edward's wrapping-up process: the "Belt Buckle" trick.

\end{ack}


\section{Preliminaries, Lipschitz Deformation Principle and Examples}
\label{section preliminaries and lipschitz deformation principle}

First, we convene that, throughout this paper, \emph{Lipschitz} will be understood to be a local condition.

\begin{definition}
\label{definition imp}
Let $f\colon X\rightarrow Y$ be a map between metric spaces.  The map $f$ is \emph{Lipschitz} if for every point in $x$, there exists an open set $U$ of $x$ and a constant $L_x$, such that for every $x',x''\in U$, 
\[ d_{Y}(f(x'),f(x''))\leq L_xd_{X}(x',x'').\]

The map $f$ is a \emph{Lipschitz embedding} if $f\colon X\rightarrow f(X)$ is Lipschitz and $f^{-1}\colon f(X)\rightarrow X$ exists and is Lipschitz. In particular, $f$ is a Lipschitz isomorphism onto its image. A map $f$ is a \emph{Lipschitz immersion} if it is locally a Lipschitz embedding. 
\end{definition}

A useful result that we will use, along with its variations, often and tacitly is the following (see \cite{siebenmann1972deformation}).

\begin{prop}
Let $h\colon F\rightarrow F'$ be an open embedding between locally compact and locally connected $T_2$ spaces. Let $C$ be a compactum in $F$. Assume $g\colon F\rightarrow F'$ is an open embedding sufficiently near $h$. Then, $h(C)\subseteq g(F)$. Further, if $g=h$ outside $C$, and $g$ is sufficiently near $h$, then, $g(F)=h(F)$.

\end{prop}

Now consider the following deformation rule.

\begin{definition}[Lipschitz Deformation Principle]
Let $X$ be a metric space (typically assumed to be locally compact and locally connected). Then we denote by $\mathcal{L}(X)$ (a statement that may or may not hold on $X$) the following:

$\mathcal{L}(X):$ For $U$ open in $X$ and $C$, a compactum in $U$. The following always holds. 

$\mathcal{L}(X;C;U):$ For every open Lipschitz embedding $h\colon U\rightarrow X$ sufficiently near the inclusion (in the compact-open topology), there exists a Lipschitz isomorphism $h'\colon X\rightarrow X$ such that $h'=h$ on $C$ and $h'=i$ (the inclusion) away from $U$. Further, the rule $h\rightarrow h'$ is \emph{canonical}. That is, it is a continuous function of $h$ for $h$ sufficiently near $i$. Moreover it sends the inclusion $i\colon U\rightarrow X$ to the identity $\mathrm{id}\colon X\rightarrow X$. 

\end{definition}

\begin{exmp}
The principle $\mathcal{L}(X)$ is a local principle for locally compact, and locally connected metric $X$. For example, if $X$ is a Lipschitz manifold, then $\mathcal{L}(X)$ holds true \cite{Sullivan}. To show this, it suffices to show that $X$ can be covered by open sets on which Sullivan's principle holds true. In particular, if $(X,g)$ is a connected compact Riemannian manifold, then with the induced length distance, $\mathcal{L}(X)$ holds true.

\end{exmp}

\begin{exmp}
\label{locallyfinitesimplicialcomplexes}
If $X$ is a locally finite simplicial complex (with the barycentric metric) then $\mathcal{L}(X)$ holds true \cite{SiebenmannSullivan}.  We note that given any compact Alexandrov space $X$, one can associate to it a simplicial complex, for which $X$ has the same Lipschitz homotopy type as the complex. More precisely, every open cover $\mathfrak{U}$ of $X$, admits a refinement $\mathfrak{U}'$, such that the nerve of $\mathfrak{U}'$, which we denote by $\mathfrak{N}(\mathfrak{U}')$, has the same Lipschitz homotopy type as $X$ \cite{goodcoverings}. Such refinements form "good coverings" for Alexandrov spaces (see \cite{goodcoverings}). In fact, the nerve complex has been useful to obtain several interesting results. For example, using the nerve complex, one can obtain Lipschitz homotopy finiteness results in the setting of Alexandrov spaces \cite{lipschitzhomotopyconvergence}. For further references we refer the reader to \cite{locallipschitzcontractibility,fujioka2023lipschitzhomotopyconvergencealexandrov,applicationsofgoodcoverings}.

\end{exmp}

\begin{exmp}
 If $X^n$ is an $n$-dimensional Alexandrov space, then there exists a connected open dense subset $M$ of $X$ that is a Lipschitz manifold. That is, every point $p$ of $M$ has a neighborhood bi-Lipschitz homeomorphic to an open region in $\mathbb{R}^n$  \cite{burago1992ad}. In particular, $\mathcal{L}(M)$ holds true. We intend to show in a subsequent work that $\mathcal{L}(X)$ holds true. We note that if $n=1$ or $2$, then $X^n$ admits a Lipschitz structure since, by the work of Perelman \cite{perelman1991alexandrov}, every Alexandrov space of dimension one and two is a topological manifold, and every such manifold admits a Lipschitz structure by the work of Sullivan \cite{Sullivan}. If $n=3$, then for any $p\in X^3$, $\Sigma_p(X)$, the space of directions at $p$, is a two-dimensional Alexandrov space, and hence admits a Lipschitz structure. For higher dimensions, the situation is not as straightforward, since non-manifold Alexandrov spaces become more common as dimension increases. For further references concerning the structure of low dimension Alexandrov spaces, we refer the reader to \cite{fernandoandluisazting,Fernando3,Fernando4, 3dsurvey, nunezzimbron,nunezzimbron3,lytchakricci,Torusactions,garcia-zarei,amann2019equivariantcohomologycohomogeneityalexandrov,mitsuishi2024collapsingthreedimensionalalexandrovspaces,mitsuishiandyamaguchi,shioyavolumecollapsedthree}.

\end{exmp}

\begin{exmp}
Note that neither Sullivan's deformation principle $\mathcal{L}$ nor Siebenmann's principle $\mathcal{D}$ are principles that one can take for granted. For example, as observed by Lytchak and Nagano \cite{Lytchak-Nagano}, there exists a 2 dimensional GCBA space $X$ (locally geodesically complete with curvature bounded by above), due to Kleiner \cite{kleiner}, such that $X$ admits a point with no conical neighborhood. Such spaces share many structural properties to Alexandrov spaces (see \cite{Lytchak-Nagano,LytchakNagano2}). For further references concerning the structure of spaces with upper curvature bounds, we refer the reader to \cite{lytchak3,lytchak5,lytchak6,lytchak7,lytchak9, stadlerI, stadlerII,  tadashi4,tadashi}.

\end{exmp}

\begin{exmp}
Given the previous examples, it is natural to wonder about the situation when a closed Alexandrov space $(X,d)$ admits, simultaneously, an upper and lower curvature bound, in the sense of Toponogov. In this case, $X$ is a space of bounded curvature \cite{NikolaevI,NikolaevII}. In which case, according to \cite{nikolaev}, $X$ admits a canonical Riemannian structure with a smooth atlas of regularity $C^{3,\alpha}$ for $\alpha \in (0,1)$. In fact, Nikolaev's approximation theorem \cite{nikolaev} asserts that any such space $X$ with curvature $K_1\leq \mathrm{curv}\leq K_1'$,  for every $K_2'>K_1'$ and $K_2<K_1$, one can find Riemannian metrics $g_i$ with $K_2\leq \mathrm{sec}_{g_i}\leq K_2'$  such that $(X,d_{g_i})$ converges, in the Lipschitz topology, to $(X,d)$.

\end{exmp}

\begin{exmp}
Complementing the previous examples, it is worth mentioning that some geodesic spaces with lower curvature bounds do not even admit an open dense subset that is a manifold and hence, neither Siebenmann nor Sullivan's theory applies immediately, not even locally. For example, Hupp, Naber and Wang showed, among many things, that certain RCD$(K,N)$ spaces do not have an open dense subset that is a manifold \cite{hupp2023lowerriccicurvaturenonexistence}. $\mathrm{RCD}(K,N)$ spaces are metric generalizations of Riemannian manifolds with lower Ricci curvature bounds.  For the convenience of the reader, we will mention how these spaces relate to Alexandrov spaces. Petrunin \cite{PetruninRCD} (see also the work of Zhang-Zhu \cite{ZhangZhu}) showed that $n$-dimensional Alexandrov spaces with curvature bounded below by $K$ satisfy the so called CD$((n-1)K,n)$ curvature condition. Combining these results with the work of Kuwae-Machigashira-Shioya \cite{Kuwae-Machigashira-Shioya}, it follows that Alexandrov spaces are RCD spaces (see also \cite{gigli2,non-collapsedrcdspaces}).  Later, Lytchak and Stadler \cite{lytchak7} proved the converse. Namely, they showed that if $(X,d,\mathcal{H}^2)$ is an RCD$(K,2)$ space, then it is an Alexandrov space. In particular, they settled a conjecture raised by Villani \cite{Villani}. The structure of $\mathrm{RCD}$ and related spaces is a very active field of research. See, for example, \cite{villani2009optimal, guijarro2016isometry, jiayinandweiexamples,CheegerColdingI,CheegerColdingII,CheegerColdingIII,JiayinWang,SemolaII,SemolaIII,Semola4,NaberCheegerJiang,Naber2,GarciaKellMondinoSosa,KapovithcXingyu,Deng,gigliandenrico,Zamorafiniteness,guijarro2016isometry,fundamentalgroupsofrcdspacesjaimesergio,invariantmeasuresandlowerrricci,sosa,universalcoverofrcspaces,HondaII,HondaI}.

\end{exmp}

Now that we have given various examples arising from different contexts, we return to the theory of deforming Lipschitz homeomorphisms. To that end, we first make the following remark.

\begin{remark}
    Assume $X$ is locally compact and locally connected.  Then $\mathcal{L}(X)$ holds true if and only if the following holds true. For $U$ open in $X$, and $C,C'$, compacta in $U$, such that $C'$ is a neighborhood of $C$, the following always holds. $\mathcal{L}(X;C,C',U):$  For every open Lipschitz embedding $h\colon U\rightarrow X$ sufficiently near the inclusion, the Lipschitz isomorphism $h'$ arising from the rule $h\rightarrow h'$ can be taken to be the inclusion outside $C'$.

\end{remark}

We conclude this section by discussing some conventions we will follow in this paper.

\subsection*{Conventions} Whenever we deal with $h$ and the induced map $h'$ arising from a deformation principle, we will always assume that $h'$ satisfies the properties mentioned in the principle. Also, we will often write $h$ for $h|$  (where $h|$ denotes $h$ restricted to a set that is usually open).

\section{Lipschitz Isotopies}
\label{section Lipschitz Isotopy theorem}

In this section, we shall prove a Lipschitz analogue of the isotopy extension theorem in \cite{siebenmann1972deformation}. Our proof is slightly more technical than Siebenmann's. This is due to a couple of reasons: First, as stated, Sullivan's deformation principle is slightly different than Siebenmann's $\mathcal{D}(X)$. Second, our statement offers a mild strengthening of the corresponding one in Siebenmann's paper.  That is, Siebenmann proves the isotopy extension theorem for when $f_b$ is the inclusion and does not immediately address whether $F_t\circ f_b=f_b$ away from a neighborhood of $C$ (when $C$ is compact). Third, we employ a ``thickening technique'' by controlling the borders of the sets so as to ensure that we can glue Lipschitz maps to a larger Lipschitz map (recall that our definition of Lipschitz is local). This idea will be useful, and used throughout the paper. In particular, it will be used in the proof of the main theorems. To the best of our knowledge, this technique does not appear in the literature.  Lastly, we note that we will also use ideas from \cite{siebenmann1972deformation}.

 To set the stage, we will first show that the Lipschitz deformation principle $\mathcal{L}(X)$ implies a relative deformation principle. Consequences of relative deformation principles are important and abound. For example, in the topological category, the "strong gluing theorem", a relative version of the important "gluing theorem" (see \cite{Vitali}) allowed Perelman to prove his stability theorem. Siebenmann used relative versions of his deformation principle to obtain deformation results, including but not limited to, the theory of foliations \cite{siebenmann1972deformation}. Edward's and Kirby \cite{DeformationsofSpaces} used relative versions of their deformation theory, to show, among many things, that the homeomorphism groups of certain manifolds is locally contractible in a relative manner.

\begin{prop}
Assume $X$ is a compact locally connected metric space such that $\mathcal{L}(X)$ holds. Then for $U$ open in $X$, $A,A'$ closed in $X$ such that $A'$ is a neighborhood of $A$, and $B$ compact in $U$, the following always holds:

$\mathcal{L}(X,A,A',B,U):$ For every open Lipschitz embedding $h\colon U\rightarrow X$ sufficiently near the inclusion (in the compact open topology) and such that $h$ is the inclusion on $A'\cap U$, then there exists a Lipschitz isomorphism $h'\colon X\rightarrow X$ such that $h'=h$ on $B$ and $h'=i$ on $X-U$ and $A$.

\end{prop}

\begin{proof}
Find small open sets $U_1$ of $B$ in $U$ and $U_2$ of $X-U$ that are disjoint. Find an open $V$ such that $\overline{V}$ is compact and such that $B\subseteq V\subseteq \overline{V}\subseteq U_1\subseteq U$. Now define $h_1\colon V\cup U_2\cup \mathring{A}'\rightarrow X$ by setting it to be $h$ on $V$ and the inclusion on $U_2\cup \mathring{A}'$. By our assumptions on $h$, $h_1$ is well defined. Moreover, for $h$ sufficiently close to the inclusion, $h_1$ is an open Lipschitz embedding. Thus, for $h$ close to the inclusion, there exists a Lipschitz isomorphism $h'\colon X\rightarrow X$ such that $h'=h_1$ on $B\cup (X-U)\cup A$ and $h'$ is the inclusion away from $V\cup U_2\cup \mathring{A}'$.

\end{proof}

\begin{proof}[Proof of Theorem \ref{lipschitz isotopy}]
Fix $b\in B$ and consider $f_t'=f_t\circ f_b^{-1}\colon f_b(V)\rightarrow X$. Put $C'=f_b(C)$ and $V'=f_b(V)$. Clearly, $\partial{C}'$ is compact. Thus, one can find a set $D$, closed in $X$, and such that $C'\subseteq \mathring{D}\subseteq D\subseteq V'$. Further, one can choose $D$ so that $\partial{D}$ is compact. Since $\partial{D}$ is compact, and $V'$ is locally compact, it follows that there is an open set $U$, with compact closure so that $\partial{D}\subseteq U\subseteq \overline{U}\subseteq V'-C'$. Since $\partial{D}$ is compact, we may find compacta $K_1,K_2,K_3,K_4$, such that $K_i\subseteq \mathring{K}_{i+1}$, $K_4\subseteq U$ and $\partial{D}\subseteq \mathring{K}_1$. At last, find a closed neighborhood $Z$ of $C'$ in $\mathring{D}$ so that $Z\cap \overline{U}=\varnothing$. Observe that as $\overline{U}$ is compact, for $t$ sufficiently close to $b$, one has $f_t'(\overline{U})\subseteq V'-C'$. Provided $t$ is close enough to $b$, there exists a Lipschitz isomorphism $h_t'\colon X\rightarrow X$ such that $h_t'=f_t'$ on $K_3$ and $h_t'=i$ (the inclusion) away from $K_4$. Now, considering $h'_t|_{U}$, for $t$ sufficiently close to $b$, the map $g_t=f_t' \circ h_t'|_{U}^{-1}\colon U\rightarrow X$ is well defined. Indeed, for $t$ close enough to $b$, $g_t(U)=f_t'(U)$, $g_t=f_t'$ away from $K_4$ and $g_t$ is the inclusion on $K_1$. Now, find  an sufficiently small open set $O$ of $\partial{U}$ in $V'$ such that $O\cap K_4= \varnothing$. Thus, for $t$ sufficiently close to $b$, define $H_t\colon V'\rightarrow X$ to be
\[ H_t = \begin{cases} 
      f_t' &  (V'-U)\cup O \\
      g_t & U \\

   \end{cases}
\]

Observe that if $x\in O\cap U$, then $g_t(x)=f_t'(x)$ by construction of $g_t$. Hence $H_t$ is well defined and is an open Lipschitz immersion. Now we will check that for $t$ close to $b$, $H_t$ is an Lipschitz embedding. Indeed, recall that $g_t(U)=f_t'(U)$ for $t$ close to $b$. In which case, if $x=f_t'(w)$ for some $w\in O\cap U$. Then, $f_t'(w)=g_t(w)$. Hence, $f_t'^{-1}(x)=g_t^{-1}(x)$. Since $B$ is locally connected, for $t$ close to $b$, $H_t(\mathring{D})=\mathring{D}$ \cite{siebenmann1972deformation}. Thus, for $t$ close enough to $b$, define $F_t\colon X\rightarrow X$ by setting it to be $H_t$ on $\mathring{D}$ and the inclusion on $\mathring{K}_1\cup (X-D)$. Arguing as in the preceding paragraph, the map $F_t\colon X\rightarrow X$ is a Lipschitz isomorphism, and satisfies the desired conditions. In particular, $F_t\circ f_b=f_t$ on $f_b^{-1}(\mathring{Z})$. Lastly, observe that if $C$ was compact, then $Z$ and $D$ can be chosen to be compact. 
\end{proof}

\begin{remark}
\begin{enumerate}
\
    \item To prove the Lipschitz union lemma, Lemma \ref{uniion lemma}, one essentially only requires (in addition to refining the deformation principle) that $f_b$ is the inclusion.
    \item If $X$ is compact, then $F_t$ will be a bi-Lipschitz homeomorphism (globally Lipschitz with globally Lipschitz inverse).
    \item If $X$ is compact, or more generally just has finitely many components, and $B$ is not necessarily locally connected, then as in \cite{siebenmann1972deformation}, Theorem \ref{lipschitz isotopy} still holds.
    \item As in \cite{siebenmann1972deformation}, if $B=I^n$ then  one may take $N_b$ to be $I^n$, the $n$-cube.
\end{enumerate}
\end{remark}

\section{Lipschitz Submersions and Lipschitz Union Lemma}
\label{section Relative Deformations, Lipschitz Submersions and Lipschitz Union Lemma}

The general framework of the deformation theory in the Lipschitz category does not work as briefly as it does in the topological category. Let us clarify why this is the case. In the topological category, we have the following result \cite{siebenmann1972deformation} that is quite useful in the topological deformation theory.

\begin{prop}\label{proposition useful in top category}

Assume $F$, $F'$ and $B$ are metric spaces. Assume $f\colon B\times F\rightarrow B\times F'$ is continuous map that respects the projection onto the $B$ factor. Assume for each $t\in B$, the map $f_t(x)=proj_{F'}f(t,x)$ is an open embedding, where $proj_{F'}\colon B\times F'\rightarrow F'$ denotes the projection map onto $F'$. If $F'$ is locally compact, and $B$ is locally connected, then $f$ is an open embedding.

\end{prop}

\begin{remark}
A more general version of the above proposition holds true (see \cite{siebenmann1972deformation}). 
\end{remark}

In general, in Theorem \ref{lipschitz isotopy}, we only have a continuous family of open embeddings $f_t$ that can be perturbed to another continuous family of Lipschitz isomorphisms $F_t$. Basic examples show that it need not be the case that the perturbed family will inherit  stronger regularity properties as the following example more clearly articulates.

\begin{exmp}If $F$ is a map of the form $F(x,t)=(f_t(x),t)$, then $F$ need not be a Lipschitz isomorphism, even if each $f_t$ is.  Therefore, Proposition \ref{proposition useful in top category} does not hold in the Lipschitz category.

\end{exmp}

Due to the preceding example, in order to obtain a Lipschitz deformation theory that is parallel to Siebenmann's, it is customary to refine the Lipschitz deformation deformation principle. To that end, one must define the notion of a Lipschitz isotopy \cite{SiebenmannSullivan}.

\begin{definition}
A \emph{Lipschitz isotopy} $F\colon I\times X\rightarrow I\times Y$, where $I,X$ and $Y$ are metric is an open embedding that respects the projection on the first coordinate and is a Lipschitz isomorphism onto its image.
\end{definition}

Due to the issue raised above, one makes the following addendum to the deformation statement $\mathcal{L}(X)$ \cite{SiebenmannSullivan}:

\begin{definition}[Refined Principle]
\label{refined principle}
$\mathcal{L}^{*}(X)\colon $ For $U$ open in $X$, and $C$ compactum in $U$, one appends to $\mathcal{L}(X;C;U)$ the following property.

$(P)\colon$ For $B$ a metric space, if $h_t\colon U\rightarrow X$, for $t\in B$ gives a Lipschitz isotopy such that the rule $h_t\rightarrow h_t'$ is well defined for all $t\in B$, then $h_t'\colon X\rightarrow X$ for $t\in B$ gives a Lipschitz isotopy.

\end{definition}
The above property is natural. For example, it is satisfied by Lipschitz manifolds \cite{Sullivan} and locally finite simplicial  complexes \cite{SiebenmannSullivan}. The (refined) principle, which we will denote by $\mathcal{L}^{*}(X)$ further yields the following results, a strengthening of the isotopy extension principle and a deformation which respects the ``bundle'' structure (cf. Theorem 6.1 in \cite{siebenmann1972deformation}). The proof of Theorem \ref{lipschitz isotopy}, combined with property $(P)$ yields the following corollaries.

\begin{corollary}
\label{strengthening}
Assume $B$ is a locally connected metric space and $X$ is a metric space that is locally compact and locally connected. Assume $V$ is a topological space such that $f_t\colon V\rightarrow X$ for $t\in B$ is a family of open Lipschitz embeddings giving a Lipschitz isotopy. Assume further that $C$ is a closed subset of $V$. Assume further that for all $t\in B$, $f_t(C)$ is closed in $X$, and that $\partial{C}$ is compact in $V$. If $\mathcal{L}^{*}(V-C)$ holds. Then for each $b\in B$, there exists a neighborhood $N_b$ of $b$, and a family of Lipschitz isomorphisms $F_t\colon X\rightarrow X$ $(t\in N_b)$ giving a Lipschitz isotopy such that $F_t\circ f_b=f_t$ near $C$. Further, if $C$ is compact then $F_t\circ f_b=f_b$ away from a compact neighborhood of $C$.
\end{corollary}

\begin{corollary}
Assume $X\times B$ is a metric product that is locally compact, locally connected and $\mathcal{L}^{*}(X)$ holds. Assume $U=U_1\times U_2$ is an open subset of $X\times B$, and $C=C_1\times C_2$ is a compactum in $U$. If $h\colon U\rightarrow X\times B$ is an open Lipschitz embedding sufficiently near the inclusion and such that $h$ respects the projection onto the $B$ factor. Then, there exists an open Lipschitz embedding $h'\colon U\rightarrow X\times B$ such that $h'=i$ on $C$ and $h'$ is $h$ away from a compactum $K$ in $U$. Furthermore, $h'$ can be chosen so that it respects the projection onto the $B$ factor.
\end{corollary}

\begin{proof}
Write $h(x,t)=(h_t(x),t)$. Provided $h$ is close enough to the inclusion, it follows that there exists a Lipschitz isomorphism $h_t'\colon X\rightarrow X$ such that $h_t'=h$ on $C_1$ and $h_t'$ is the inclusion away from $U_1$. Hence, by assumption, $h'(x,t)=(h_t'(x),t)$ is a Lipschitz isotopy. Now, one can modify $h'$ to obtain the desired map.

\end{proof}

Now we will introduce the notion of a \emph{Lipschitz submersion}  (c.f. Definition 6.8 in \cite{siebenmann1972deformation}, and page 59 in \cite{kirby-siebenmann}). Note that topological submersions enjoy many properties and have shown to be useful (see \cite{kirby-siebenmann}). They are, in particular, a weaker form of a fiber bundle. We devote the rest of this paper to showing that a proper Lipschitz submersion with fibers satisfying $\mathcal{L}^{*}(X)$ is indeed a Lipschitz fiber bundle. In what follows, we will endow the product space with the sum metric (of course, one may also choose another equivalent one).

\begin{definition}
Let $E$ and $X$ be be metric spaces. A Lipschitz map $p\colon E\rightarrow X$ is a \emph{Lipschitz submersion} if for each $y\in E$ there exists an open subset $U$ of $y$ in $F=p^{-1}p(y)$ and a neighborhood $N$ of $p(y)$ in $X$ and an open Lipschitz embedding $f\colon U\times N\rightarrow E$ onto a neighborhood of $y$ such that the following holds true.
\begin{enumerate}
    \item $p\circ f\colon U\times N\rightarrow N$ is the projection map.
    \item $f(u,p(F))=u$ for all $u\in U$.
\end{enumerate}

We will refer to the map $f\colon U\times N\rightarrow E$ as a \textit{Lipschitz product chart about} $U$ for $p$.

\end{definition}
\begin{exmp}
A bi-Lipschitz homeomorphism is a Lipschitz submersion.

\end{exmp}

\begin{exmp}
Projection maps are Lipschitz submersions.

\end{exmp}

\begin{exmp}
If $X^n$ is a compact Alexandrov space and $f\colon X^n\rightarrow \mathbb{R}^n$ is an admissible map, then near its regular points it is a Lipschitz submersion.

\end{exmp}

 Now we shall give a proof of the Lipschitz analogue of Lemma 6.12 in \cite{siebenmann1972deformation}. The following lemma was used to give a proof of the Union lemma in \cite{siebenmann1972deformation}. We remark that there is no proof there. We give a proof here only to further illustrate the usefulness of property $(P)$. Our proof of the union lemma is for the most part different than the one in \cite{siebenmann1972deformation}.

\begin{lemma}
Assume $F$ and $B$ are metric spaces, let $p\colon F\times B\rightarrow B$ denote the projection map. Fix $b\in B$. Identify $F$ with $p^{-1}(b)$ in the obvious way. Let $U$ be an open subset of $F$ such that $\overline{U}$ is compact in $F$. Let $C\subseteq U$ be a compact set in $U$ and  let $f\colon U\times N\rightarrow F\times B$ be a product chart about $F$. Assume $F$ is locally compact, locally connected and $\mathcal{L}^{*}(F)$ holds. Then there exists a neighborhood $N'$ of $b$ and a product chart $g\colon F\times N'\rightarrow F\times B$ about $F$ such that $g=f$ near $C\times b$ and further,  $g=\mathrm{id}$ outside $K\times N'$, where $K$ is some compact neighborhood of $C$.
\end{lemma}

\begin{proof}
Define $f_t\colon U\rightarrow F$ to be an open Lipschitz embedding such that, $f(u,t)=(f_t(u),t)$ for all $u\in U$ and $t\in N$. Find a neighborhood $N$ of $b$  family $F_t\colon F\rightarrow F$ of Lipschitz isomorphisms, giving a Lipschitz isotopy such that $F_t=f_t$ near $C$  and further, $F_t=\mathrm{id}$ (identity) away from a compact neighborhood $K$ of $C$. Now, define $g$ to be $g(x,t)=(F_t(x),t)$.
\end{proof}

\begin{lemma}[Lipschitz Union Lemma]
\label{uniion lemma}
Let $p\colon E\rightarrow X$ be a Lipschitz submersion, where $X$ is  locally compact and locally connected. Assume $F=p^{-1}(b)$, $b=p(y)$  is locally compact, locally connected, $U$ and $V$ are open subsets of $F$ and that we have Lipschitz product charts $f\colon U\times N_1\rightarrow E$ and $g\colon V\times N_2\rightarrow E$ about  $U$ and $V$ respectively, for $p$. Assume $U$ and $V$ are open neighborhoods of compacta $A$ and $B$ in $F$ (respectively) and that $\mathcal{L}^{*}(F)$ holds true. Then there exists a Lipschitz product chart $h\colon W\times N\rightarrow E$, where $W$ is an open neighborhood of $A\cup B$ in $F$  and such that $h=f$ near $A\times b$ and $h=g$ near $(B-U)\times b$. 
\end{lemma}

\begin{remark}
In applications, we will concern ourselves with the case where $F$ is compact.
\end{remark}

\begin{proof}
Assume $A\cap B\neq \varnothing$. Find compacta $A_1,A_2,A_3,A_4$ in $U$ and $B_1,B_2,B_3,B_4$ in $V$ such that $A\subseteq \mathring{A}_1\subseteq A_1\subseteq \mathring{A}_2\subseteq A_2\subseteq \mathring{A}_3\subseteq A_3 \subseteq \mathring{A}_4\subseteq A_4\subseteq U$ and $B\subseteq \mathring{B}_1\subseteq B_1\subseteq \mathring{B}_2\subseteq B_2\subseteq \mathring{B}_3\subseteq B_3\subseteq \mathring{B}_4\subseteq B_4\subseteq V$.  Put $W_0=\mathring{A}_2\cap \mathring{B}_2$ and $K=A_3\cap B_3$. Now find an sufficiently small open neighborhood $O$ of $\partial{W}_0$, with compact closure such that $\overline{O}\subseteq U\cap V$. Since $f$ and $g$ are Lipschitz product charts, for $N_3$, a sufficiently small neighborhood of $b$, we can consider the composition $g^{-1}\circ f\colon (\mathring{A}_4\cap \mathring{B}_4)\times N_3\rightarrow (U\cap V)\times (N_1\cap N_2)$. Since $g^{-1}\circ f$ is a Lipschitz isotopy, there exists a neighborhood $N_4$ of $b$ and a bijective Lipschitz isotopy $G:F\times N_4\rightarrow F\times N_4$ such that $G=g^{-1}\circ f$ on $\overline{W}_0\times N_4$ and $G=\mathrm{id}$ away from $K\times N_4$.  Hence, for $N_4$, a very small enough neighborhood of $b$, we can construct a Lipschitz product chart $\varphi\colon V\times N_4\rightarrow E$ as follows. Put $Z_1=((V-\overline{W}_0) \cup O)\times N_4$  and $Z_2= W_0\times N_4$. Define $\varphi$ by
 \[ \varphi=\begin{cases} 
      g \circ G& Z_1 \\
      f &   Z_2 \\
   \end{cases}
\]

Now we will show that $\varphi$ is well defined and is a Lipschitz product chart provided $N_4$ is small enough. It is clear that, provided $N_4$ is small enough, $\varphi$ is well defined. Further, observe that if $e=gG(x,t)=f(w_0,s)$, where $(x,t)\in Z_1$ and $(w_0,s)\in Z_2$ then $G^{-1}g^{-1}(e)=(x,t)=G^{-1}g^{-1}f(w_0,s)=(w_0,s)=f^{-1}(e)$. Now we follow  \cite{siebenmann1972deformation}. Put $W= \mathring{A}_1\cup \mathring{B}_1$. For $N=N_4$ small enough, define $h\colon W\times N\rightarrow E$ by setting it to be $f$ on $\mathring{A}_1\times N$ and $\varphi$ on $\mathring{B}_1\times N$. Clearly $h$ is an open Lipschitz immersion. Furthermore, if $W$ and $N$ are small enough, then $h$ will be an embedding. Hence, the result follows.
\end{proof}

Now, we will prove Theorem \ref{lipschitz submersions}. The proof is similar to the corresponding one in \cite{siebenmann1972deformation}. We make a simplification.

\begin{proof}[Proof of Theorem \ref{lipschitz submersions}]
For each $y\in C$, there exists a sufficiently small neighborhood $U_y$ of $y$ in $p^{-1}(b)$  containing $y$ and a neighborhood $N_y$ of $b$ in $X$ and a Lipschitz product chart $f_y:U_y\times N_y\rightarrow E$.  Since $C$ is compact, there are only finitely many $U_{y_1},...,U_{y_n}$. Thus, the union lemma implies that there exists an open neighborhood $W$, where $W$ contains $C$ and an Lipschitz product chart $h\colon W\times N\rightarrow E$. Hence, if $p^{-1}(x)$ is compact, then there exists a Lipschitz product chart $h\colon p^{-1}(x)\times N\rightarrow E$ for $p$. It remains to find a neighborhood $N'$ of $x$ in $X$ such that $h(p^{-1}(x)\times N')=p^{-1}(N')$. Indeed, since $h$ is an open map, it follows that $h(p^{-1}(x)\times N)$ is open in $E$. Thus, $E-h(p^{-1}(x)\times N)$ is closed in $E$. Furthermore, $p(E-h(p^{-1}(x)\times N))$ is closed in $X$. Note that since $h$ is a product chart, $x$ is not in $A=p(E-h(p^{-1}(x)\times N))$. Thus, set $N'=N\cap (X-A)$. Clearly $N'$ is an open set in $X$. What is more, since both $N$ and $(X-A)$ contain $x$, it follows that $N'$ is an open neighborhood of $x$. It remains to check that $h(p^{-1}(x)\times N')=p^{-1}(N')$. By definition of product chart, it suffices to check the containment $p^{-1}(N')\subseteq h(p^{-1}(x)\times N')$. Indeed, if $w\in p^{-1}(N')$ then $p(w)\in N'$. Thus, $p(w)\notin A$ and so, $w\in h(p^{-1}(x)\times N')$. The proof of the theorem is now complete.
\end{proof}

\section{Applications}
\label{section applications}
\subsection{Perelman's Deformation Lemma and Lipschitz Gluing Theorem}

The topological gluing theorem is, along with the topological fibration theorem, crucial ingredients for the topological stability theorem. In particular, as mentioned in the introduction, the gluing theorem gives mild conditions as to when one can glue local homeomorphisms near a approximation, to a global homeomorphism that remains near the approximation. In this section we shall prove the Lipschitz analogue. The deformation lemma (Lemma 4.7 in \cite{Vitali} and Assertion $1$ in \cite{perelman1991alexandrov}) is the technical topological result that is used in the proof of the gluing theorem.  Indeed, the topological gluing theorem follows from it. Note that in the deformation lemma, one doesn't really require that $X$ is an $MCS$ space, just that it is locally compact, and satisfies Siebenmann's deformation principle $\mathcal{D}(X)$. 

Before proceeding, we first comment on the proof of the gluing theorems. 
In the topological category, Perelman used the "deformation lemma".  In the Lipschitz category, due to the nature of Sullivan's deformation principle, one can make do without this lemma.

To ease readability, we shall use the the terminology as in Perelman's original argument. To that end, we denote by $\chi$ various positive increasing continuous functions defined for sufficiently small arguments.

Although the deformation lemma is not strictly required for the proof of the Lipschitz gluing theorem, in this section, we shall first prove the Lipschitz analogue of the deformation lemma.  The reasons for this are two fold. First, we give a proof for the sake of completeness and convenience of the reader. Second, it is not clear to us why in the original argument (see proof of Assertion $1$ in \cite{perelman1991alexandrov}), in the notation of Perelman's, that $\varphi_1$ is well defined on $V\backslash W$ (it is defined on $U\backslash \overline{W}$).  Thus our proof is slightly different.

\begin{prop}[Lipschitz Deformation Lemma]

\label{Lipschitz Deformation Lemma}
Let $X$ be a locally compact, locally connected metric space such that $\mathcal{L}(X)$ holds. Let $W,V,U$ be open sets satisfying $\overline{W}\subseteq V\subseteq \overline{V}\subseteq U$. If $h\colon U\rightarrow X$ is an open Lipschitz embedding sufficiently near the inclusion. That is, $h$ is $\delta$-close to the inclusion $i$ (for $\delta$ a sufficiently small value), then there exists an open Lipschitz immersion $\tilde{h}\colon U\rightarrow X$, $\chi(\delta)$-close to the inclusion such that $\tilde{h}=i$ on $U\backslash V$ and $\tilde{h}=h$ on $W$. 

\end{prop}

\begin{proof} Find sufficiently small compacta $C_1,C_2,C_3,C_4$ and an open set $O\supseteq \partial{W}$, with $O\subseteq V$, such that $C_i\subseteq \mathring{C}_{i+1}$, $\partial{V}\subseteq \mathring{C_1}$,  $C_4\subseteq U\backslash \overline{W}$ and $O\cap C_4= \varnothing$. By (the proof of) Theorem \ref{lipschitz isotopy}, for $\delta$ sufficiently small, there exists an open Lipschitz embedding $h_1\colon U\backslash \overline{W}\rightarrow X$  such that $h_1=i$ on $C_3$ and $h_1=h$ away from $C_4$. Now define $\tilde{h}$ as follows:

\[ \tilde{h} = \begin{cases} 
      i &  (U\backslash V)\cup \mathring{C}_1 \\
      h_1 & V\backslash \overline{W} \\
      h & W\cup O    

   \end{cases}
\]

\end{proof}

 Our proof of the Lipschitz gluing theorem is an adaptation of Perelman's. We will follow Perelman's proof for the most part (see also \cite{Vitali}). Though, we will make some changes. The proof is similar in essence to the proof of Lemma \ref{uniion lemma}.

\begin{proof}[Proof of Theorem \ref{Lipschitz Gluing Theorem}]
We will will induct on the size of $\mathfrak{A}$. If $|\mathfrak{A}|=1$ then the result is true. Now, fix $\alpha_{1}$ and $\alpha_{2}$ in $\mathfrak{A}$. Find open sets $U_1^0,U_1^1,U_1^2,U_1^3,U_1^4$ in $U_{\alpha_1}$,  $U_2^0,U_2^1,U_2^2,U_2^3,U_2^4$ in $U_{\alpha_2}$ and a sufficiently small $O$ in $U_2^0$ such that the following holds.

\begin{enumerate}
    \item  For $k=1,2$, $U_k^4\Subset U_k^3\Subset U_k^2\Subset U_k^1\Subset U_k^0\Subset U_{\alpha_k}$.
    \item $X\backslash \bigcup_{\alpha \neq \alpha_1,\alpha_2} U_{\alpha} \subseteq U_1^4\cup U_2^4$.
    \item $\partial{(U_1^1\cap U_2^1)}\subseteq O$ and $O\cap (U_1^2\cap U_2^2)=\varnothing$.

\end{enumerate}

For $\delta$ small enough, one has $\varphi_{\alpha_2}^{-1}\varphi_{\alpha_1}(U_1^1\cap U_2^1)\subseteq U_{\alpha_2}$. Moreover, the map $\varphi_{\alpha_2}^{-1}\circ  \varphi_{\alpha_1}\colon U_1^1\cap U_2^1\rightarrow U_{\alpha_2}$ is  $3\delta$-close to the inclusion. Since $\mathcal{L}(X)$ holds true, it is easy to see that, provided $\delta$ is small enough, one obtains an open Lipschitz embedding $\psi\colon U_1^1\cap U_2^1\rightarrow U_{\alpha_2}$ that is close to the inclusion, and such that $\psi = \varphi_{\alpha_2}^{-1}\circ \varphi_{\alpha_1}$ on $U_1^3\cap U_2^3$ and $\psi= i$ on $(U_1^1\cap U_2^1)\backslash (U_1^2\cap U_2^2)$ (c.f. proof of Lemma \ref{uniion lemma}). If $\delta$ is small enough, $\psi$ extends to an open Lipschitz embedding $\psi$ on $U_2^0$ as follows. Indeed, set

\[ \overline{\psi} = \begin{cases} 
      i &  (U_2^0\backslash (U_1^1\cap U_2^1)) \cup O \\
      \psi & U_1^1\cap U_2^1 \\

   \end{cases}
\]

Observe, if $x\in O\cap (U_1^1\cap U_2^1)$ then $x\notin U_1^2\cap U_2^2$. Hence, $\psi(x)=x$. Thus, $\overline{\psi}$ is well defined. As $\psi$ is equal to the inclusion away from a compactum, it follows that if $\delta$ is small enough,  $\overline{\psi}$ is an open Lipschitz embedding. Now define $\varphi'\colon U_1^3\cup U_2^3\rightarrow \tilde{X}$ by the following rule:
\[ \varphi' = \begin{cases} 
      \varphi_{\alpha_1} &  U_1^3 \\
       \varphi_{\alpha_2}\circ \overline{\psi} & U_2^3 \\

   \end{cases}
\]

Observe that if $x\in U_1^3\cap U_2^3$ then $\overline{\psi}(x)=\psi(x)=\varphi_{\alpha_2}^{-1}\varphi_{\alpha_1}(x)$. Hence $\varphi'$ is well defined. Now we complete the proof as in (Theorem $B$ in \cite{perelman1991alexandrov},  Gluing Theorem 4.6 in \cite{Vitali}). Namely, if $\delta$ small enough, $\varphi'$ is an open Lipschitz embedding on $U_1^4\cup U_2^4$. What is more, provided $\delta$ is small enough, $\tilde{X}\backslash \bigcup_{\alpha\neq \alpha_1,\alpha_2}\tilde{{U}}_{\alpha}$ is contained in $\varphi'(U_1^4\cup U_2^4)$. Hence by induction a Lipschitz isomorphism $X\rightarrow \tilde{X}$ results. Since $\tilde{X}$ and $X$ are compact, this map is bi-Lipschitz and is a map with the properties we seek.

\end{proof}

\subsection{Remark on Alexander Isotopies}

\label{last section}
The Alexander isotopy is a useful tool in geometric topology. For instance, Kirby and Edward \cite{DeformationsofSpaces} used the Alexander isotopy to show that a compact manifold has locally contractible homeomorphism group. More generally, Siebenmann \cite{siebenmann1972deformation}, used an Alexander isotopy on cones to show that an open embedding near an inclusion, can be isotoped through a family of open embeddings to the inclusion. In this section, we will show that the Alexander isotopy can be taken to be through Lipschitz isomorphisms. The following proposition follows from Proposition 3 in \cite{gauld}.

\begin{prop}
Let $X$ be a metric space. If $X$ is compact then the topology on $cX$ (the open cone on $X$) is metrizable by a metric $d$ such that  

\begin{enumerate}
    \item $d(\alpha x,\alpha y)=\alpha d(x,y)$ for all $\alpha \in [0,\infty)$ and $x,y\in cX$.
    \item $d(\alpha y,\beta y)=|\alpha-\beta|d(v,y)$, where $v$ is the tip of the cone and $\alpha,\beta\in[0,\infty)$ and $y\in cX$.
\end{enumerate}

\end{prop}

Now we will recall the Alexander isotopy.

\begin{prop}
Assume $X$ is a compact, locally connected metric space and $h\colon cX\rightarrow cX$ is a Lipschitz isomorphism such that $h$ is the identity away from a fixed compactum in $cX$. Then there exists an isotopy $h_t$, where $0\leq t\leq 1$ consisting of Lipschitz isormophisms $h_t\colon cX\rightarrow cX$ such that $h_0=h$ and $h_1=id$ (the identity on $cX$).
\end{prop}

\begin{proof}
Define for each $t\in [0,1)$, the map $\chi_t\colon cX\rightarrow cX$ by $\chi_t(sx)=(1-t)sx$. Here, $s\in [0,\infty)$ and $x\in X$. Observe, for $s,s'\in [0,\infty)$ and $x,y\in X$, $d(\chi_t(sx),\chi_t(s'y))=(1-t)d(sx,s'y)$. Similarly, the inverse of $\chi_t$, $\chi_t^{-1}$, has Lipschitz constant $1/(1-t)$. Thus, define $h_t$ by setting it to be $\chi_th\chi_t^{-1}$ for $t\in [0,1)$ and for $t=1$, set $h_t$ to be the identity on $cX$.
\end{proof}

\section{Deformations with Lipschitz Constant Control}

\label{Lipschitz constant control}
  In this section, we will assume that $X$ is a locally compact and locally connected metric space. Sullivan's principle $\mathcal{L}(X)$, as we have defined earlier is a statement that does not take into account the Lipschitz constants.  Therefore, it is desirable to obtain an analogue of the principle that takes into account the Lipschitz constants. Let us illustrate how one might do this.  In analogy to the Edwards and Kirby theory \cite{DeformationsofSpaces}, using Sullivan's immersion device instead of the Torus, one is able to obtain natural versions of Sullivan's deformations principles in which one has more control on the Lipschitz constants (see \cite{tukia-vaisala, jouni, tukia-vaisala2,luukainenrespectfulquasiconformal,bi-Lipschitzconcordanceimpliesbi-Lipschitzisotopy} and references therein for proofs of these facts (and more) and for very nice expositions of Sullivan's theory). For further references on related work, we refer the reader to \cite{lippproximationsofembeddings,luukainenlipschitzembeddings,quasiconformalmanifolds}

\begin{prop} \cite{tukia-vaisala,tukia-vaisala2,jouni}
Let $U$ be an open subset of $\mathbb{R}^n$ and let $B$ be a compact subset of $U$ and $B'$ a compact neighborhood of $B$ in $U$. If $h\colon U\rightarrow \mathbb{R}^n$ is an open Lipschitz embedding that is sufficiently near the inclusion. That is, it is $\delta$-close to the inclusion ($\delta$ is sufficiently small) and is locally $L$-bilipschitz, then there exists constants $C=C(n),K=K(n)\geq 1$ and a $L_1=CL^{K}$-bi-Lipschitz homeomorphism $h'\colon\mathbb{R}^n\rightarrow \mathbb{R}^n$ such that $h'=h$ on $B$, $h'=i$ away from $B'$ and $h'$ is $\chi(\delta)$-close to the inclusion.

  \end{prop}

\begin{remark}
The constant $L_1$ depends only on $n$ and $L$.  In particular, $C$ and $K$ depend only on the ambient space.

\end{remark}

 Indeed, one can define the following deformation rule (see \cite{tukia-vaisala,jouni,bi-Lipschitzconcordanceimpliesbi-Lipschitzisotopy,Respectful,Sullivan}).

\begin{definition}
Let $X$ be a locally compact, locally connected metric space.

$\mathcal{L}^{'}(X)\colon$ Let $U$ be an open set in $X$, and $B$ is a compact subset of $U$. Then the following always holds.

$\mathcal{L}'(X,B,U)\colon$  For $Z$ a metric space, if $h_t\colon U\rightarrow X$, where $t\in Z$ are open Lipschitz embeddings so close to the inclusion so that the the deformation rule $h_t\rightarrow h_t'$  arising out of $\mathcal{L}(X,B,U)$ is well defined for all $t\in Z$, if $h_t$ induces a  Lipschitz isotopy $h$ such that $h$ is either $(1).$  Locally $L$-bi-Lipschitz, or $(2).$ $L$-bi-Lipschitz. Then there exists constants $C,K\geq 1$, independent of $h$ and $Z$, such that $h_t'$ induces, a isotopy $h'$ that is, respectively, either a  $(1').$ Locally $CL^K$- bi-Lipschitz isotopy or a  $(2').$ $CL^K$-bi-Lipschitz isotopy

\end{definition}

Armed with the previous proposition, and the above deformation rule, now one can obtain analogues of our results in which the Lipschitz constants are controlled. For example, we have the following isotopy extension theorem (which also, for example, yields an version of the union lemma in which the Lipschitz constants are controlled).

\begin{corollary}

\label{isotopywithconstantcontrol}
Assume $B$ is a locally connected metric space and $X$ is a locally compact locally connected metric space, and $V$ is an open subset of $X$, $K$ is a closed subset of $X$, in $V$ and with compact boundary in $V$.  The following holds.

Given a family $f_t\colon V\rightarrow X$ (where $t\in B$) of open Lipschitz embeddings, inducing a Lipschitz isotopy $f$ that is locally $L$-bi-Lipschitz for which for some $b\in B$, $f_b$ is the inclusion, $f_t(K)$ is closed in $X$ for all $t\in B$ and $\mathcal{L}'(V-K)$ holds true. Then, there exists a neighborhood $N_b$ of $b$, and a family of   Lipschitz isomorphisms $F_t\colon X\rightarrow X$ $(t\in N_b)$, such that the following holds:

\begin{enumerate}
\item The family $F_t$ induces a Lipschitz isotopy, with local bi-Lipschitz constants being uniform and of the form $CL^M$, where $C,M\geq 1$ are independent of $f$ and $B$ and depend only of $V-K$.
\item $F_t=f_t$ near $K$.

\end{enumerate}

\end{corollary}

\begin{remark}
If, in addition, $X$ is a length space, then the $F_t$ will be globally bi-Lipschitz and moreover, the constants will be of the form $CL^M$. Further note that if $X$ is compact, then the $F_t$ are all bi-Lipschitz.

\end{remark}

\begin{corollary}
Assume $F$ and $B$ are metric spaces, let $p\colon F\times B\rightarrow B$ denote the projection map. Fix $b\in B$. Identify $F$ with $p^{-1}(b)$. Let $U$ be an open subset of $F$ such that $\overline{U}$ is compact in $F$. Let $C\subseteq U$ be a compact set in $U$ and  let $f\colon U\times N\rightarrow F\times B$ be a product chart about $F$ that is locally $L$-biLipschitz. Assume $F$ is locally compact, locally connected and $\mathcal{L}^{'}(F)$ holds. Then there exists a neighborhood $N'$ of $b$ and a product chart $g\colon F\times N'\rightarrow F\times B$ about $F$ that is locally $CL^M$ bi-Lipschitz, where $C,M\geq 1$ are independent of $f$ and $L$. Further, $g=f$ near $C\times b$  and $g=id$ outside $K\times N'$, where $K$ is some compact neighborhood of $C$.

\end{corollary}

\printbibliography{}

\end{document}